\documentclass[12pt,a4paper]{article}
\usepackage{latexsym,amssymb,amsfonts,amsmath,amsthm,nccmath,enumitem,setspace,graphicx,bm,float}
\usepackage[dvipsnames]{xcolor}
\usepackage[hidelinks]{hyperref}
\usepackage[margin=2cm]{geometry}

\newtheorem{theorem}{Theorem}[section]

\newtheorem{lemma}[theorem]{Lemma}

\theoremstyle{definition}

\newtheorem{example}[theorem]{Example}

\newcommand{\B}{\mathcal{B}}

\def \deg {{\rm deg}}
\def \GDD {{\rm GDD}}

\def \leq {\leqslant}
\def \geq {\geqslant}

\def \mod{\pmod}

\let\oldproofname=\proofname
\renewcommand{\proofname}{\rm\bf{\oldproofname}}

\title{Zarankiewicz numbers near the triple system threshold}

\author{Guangzhou Chen\thanks{Henan Engineering Laboratory for Big Data Statistical Analysis and Optimal Control, School of Mathematics and Information Science, Henan
Normal University, Xinxiang, 453007, PR China}\qquad
Daniel Horsley\thanks{School of Mathematics, Monash University, VIC 3800, Australia \quad \texttt{danhorsley@gmail.com}}\qquad
Adam Mammoliti\thanks{School of Mathematics, Monash University, VIC 3800, Australia}}
\date{}

\begin{document}
\maketitle
\setstretch{1.2}

\begin{abstract}
For positive integers $m$ and $n$, the Zarankiewicz number $Z_{2,2}(m,n)$ can be defined as the maximum total degree of a linear hypergraph with $m$ vertices and $n$ edges. Guy determined $Z_{2,2}(m,n)$ for all $n \geq \binom{m}{2}/3+O(m)$. Here, we extend this by determining $Z_{2,2}(m,n)$ for all $n \geq \binom{m}{2}/3$ and, when $m$ is large, for all $n \geq \binom{m}{2}/6+O(m)$.
\end{abstract}

\begin{center}
{\small \textbf{Keywords:} Zarankiewicz problem; Zarankiewicz number; linear hypergraph; group divisible design}\vspace*{-3mm}
\end{center}

\section{Introduction}\label{S:intro}

For our purposes, a hypergraph $H$ consists of a set $V(H)$ of \emph{vertices} and a multiset $E(H)$ of \emph{edges}, where each edge is a nonempty subset of $V(H)$. The \emph{total degree} of $H$ is $\sum_{E \in E(H)}|E|$. We say $H$ is \emph{linear} if each pair of distinct vertices occurs together in at most one edge, and we say $H$ is a \emph{linear space} if each pair of distinct vertices occurs together in exactly one edge. A hypergraph or linear space is a \emph{$k$-hypergraph} or \emph{linear $k$-space} if each of its edges has size $k$.

For positive integers $s$, $t$, $m$ and $n$, the Zarankiewicz number $Z_{s,t}(m,n)$ is usually defined to be the maximum number of edges in a bipartite graph with parts of sizes $m$ and $n$ that has no complete bipartite subgraph containing $s$ vertices in the part of size $m$ and $t$ vertices in the part of size $n$. By considering this bipartite graph as the vertex-edge incidence graph of a hypergraph, this is equivalent to the maximum total degree of a hypergraph with $m$ vertices and $n$ edges in which each set of $s$ vertices is a subset of at most $t-1$ edges. In particular $Z_{2,2}(m,n)$ can be defined as the maximum total degree of a linear hypergraph with $m$ vertices and $n$ edges. The problem of finding these numbers was first posed by Zarankiewicz in \cite{Zar}.

It is well known that if $n=\binom{m}{2}/\binom{k}{2}$ for some integer $k \geq 3$ we have $Z_{2,2}(m,n) = \binom{m}{2}$ if and only if a linear $k$-space with $m$ vertices exists. A linear 3-space is often referred to as a Steiner triple system. We will refer to $n=\frac{1}{3}\binom{m}{2}$ as the \emph{triple system threshold}. \v{C}ul\'{\i}k \cite{Cul} observed that, for each $t \geq 2$, $Z_{2,t}(m,n)=(t-1)\binom{m}{2}+n$ when $n \geq (t-1)\binom{m}{2}$. In the case $t=2$, Guy \cite{Guy} extended this by completely determining $Z_{2,2}(m,n)$ for most values of $n$ above the triple system threshold. Our first main result completes this effort by determining $Z_{2,2}(m,n)$ for all $n$ above the triple system threshold.

\begin{theorem}\label{T:aboveSTS}
For positive integers $m$ and $n$ with $\frac{1}{3}\binom{m}{2} \leq n \leq \binom{m}{2}$, we have
\[Z_{2,2}(m,n) =
\left\{
  \begin{array}{ll}
    \left\lfloor\tfrac{1}{14}m(3m-4)+\tfrac{12}{7}n\right\rfloor & \hbox{if $m$ is even and $n \leq \left\lfloor\tfrac{1}{3}\binom{m}{2}+\frac{m}{3}\right\rfloor$,} \\[1.5mm]
    \left\lfloor\tfrac{1}{4}m(m-1)+\tfrac{3}{2}n\right\rfloor-1 & \hbox{if $m \equiv 5 \mod{6}$ and $n = \left\lceil\tfrac{1}{3}\binom{m}{2}\right\rceil$,} \\[1.5mm]
    \left\lfloor\tfrac{1}{4}m(m-1)+\tfrac{3}{2}n\right\rfloor & \hbox{otherwise.}
  \end{array}
\right.
\]
\end{theorem}

The third case of the conclusion of Theorem~\ref{T:aboveSTS} duplicates Guy's result, but the first two cases constitute improvements on it.

Our other main result determines $Z_{2,2}(m,n)$ when $m$ is sufficiently large and $n$ is below the triple system threshold but significantly above the linear $4$-space threshold of $n=\frac{1}{6}\binom{m}{2}$. Note that in some cases we in fact prove Theorem~\ref{T:belowSTS} for all $m$ and in others for all $m \geq 96$ (see Lemmas~\ref{L:evenZValue} and \ref{L:ColRosStiConsequence}), but it general we require the stronger hypothesis that $m$ is sufficiently large.

\begin{theorem}\label{T:belowSTS}
There is a positive integer $m_0$ such that for integers $m$ and $n$ with $m>m_0$ and $\frac{1}{6}\binom{m}{2}+\frac{m}{3}+40 \leq n < \frac{1}{3}\binom{m}{2}$, we have
\[Z_{2,2}(m,n) =
\left\{
  \begin{array}{ll}
    \left\lfloor\tfrac{1}{14}m(3m-4)+\tfrac{12}{7}n\right\rfloor & \hbox{if $m$ is even and $n > \left\lfloor\tfrac{1}{3}\binom{m}{2}-\frac{m}{4}\right\rfloor$,} \\[1.5mm]
    \left\lfloor\tfrac{1}{6}m(m-1)+2n\right\rfloor-1 & \hbox{if $m$ is odd and $n = \left\lfloor\tfrac{1}{3}\binom{m}{2}\right\rfloor-r$ for some $r \in S$,} \\[1.5mm]
    \left\lfloor\tfrac{1}{6}m(m-1)+2n\right\rfloor & \hbox{otherwise,}
  \end{array}
\right.
\]
where $S=\{1,2,3,4\}$ if $m \equiv 1,3 \mod{6}$ and $S=\{0,1,3\}$ if $m \equiv 5 \mod{6}$.
\end{theorem}

Many of the values of $Z_{2,2}(m,n)$ given by Theorems~\ref{T:aboveSTS} and \ref{T:belowSTS} have already been determined in our previous work \cite{CheHorMam}. However, that work does not determine $Z_{2,2}(m,n)$ for values of $n$ close to $\frac{1}{3}\binom{m}{2}+\frac{m}{3}$ or $\frac{1}{3}\binom{m}{2}-\frac{m}{4}$ when $m$ is even or close to $\frac{1}{3}\binom{m}{2}$ when $m$ is odd. In fact, results in \cite{CheHorMam} indicate that the behaviour of $Z_{2,2}(m,n)$ can be particularly interesting at points like these. Furthermore, in \cite{CheHorMam}, values of $Z_{2,2}(m,n)$ are only determined for large values of $m$, while Theorem~\ref{T:aboveSTS} applies for all values of $m$.

A large amount of work has been devoted to determining bounds on, and the asymptotic behaviour of, $Z_{s,t}(m,n)$ in various regimes of $s,t,m,n$ (see \cite{KovSosTur} for the classical upper bound and \cite{AloRonSza,AloMelMubVer,Con} for a sampling of more recent work). 
Fewer results have been concerned with determining $Z_{s,t}(m,n)$ exactly. We have already mentioned the results of \v{C}ul\'{\i}k \cite{Cul} and Guy \cite{Guy}. Roman \cite{Rom} established a family of upper bounds on $Z_{s,t}(m,n)$. In the case $s=t=2$, two of these bounds are central to our work here and we have already seen them feature in Theorems~\ref{T:aboveSTS} and ~\ref{T:belowSTS} (see Section~\ref{S:prelim} for more details). Goddard, Henning and Oellermann \cite{GodHenOel}, Collins, Riasanovsky, Wallace and Radziszowski \cite{ColRiaWalRad}, and Tan \cite{Tan} determined exact values of $Z_{s,t}(m,n)$ for various small parameters. Dam\'{a}sdi, H\'{e}ger and Sz\H{o}nyi \cite{DamHegSzo} found a number of interesting exact values and bounds on $Z_{2,2}(m,n)$ relating to finite projective planes and other designs. Finally, our own previous work \cite{CheHorMam} gives exact values for $Z_{2,t}(m,n)$ in many cases where $m$ is large and $n=\Theta(tm^2)$.

In the course of proving Theorems~\ref{T:aboveSTS} and \ref{T:belowSTS}, we give a characterisation of the existence of $3$-GDDs of type $2^u4^v$ which may be of independent interest. A \emph{$3$-$\GDD$ of type $2^{u}4^{v}$} is a linear space with $2u+4v$ vertices whose edge multiset admits a partition $\{\B,\mathcal{G}\}$ such that each edge in $\B$ has size $3$, $\mathcal{G}$ contains exactly $u$ edges of size $2$ and $v$ edges of size $4$, and the edges in $\mathcal{G}$ form a partition of the vertex set.

\begin{theorem}\label{T:3GDD4s2s}
Let $u$ and $v$ be nonnegative integers. There exists a $3$-$\GDD$ of type $4^u2^v$ if and only if
\begin{itemize}[itemsep=0mm,parsep=0mm,topsep=1mm]
    \item[\textup{(i)}]
$u \equiv 0,1 \mod{3}$ and $v \equiv 0,1 \mod{3}$; and
    \item[\textup{(ii)}]
$u \equiv 0 \mod{3}$ or $v \equiv 0 \mod{3}$.
\end{itemize}
\end{theorem}

\section{Preliminaries and upper bounds}\label{S:prelim}

Let $H$ be a hypergraph. If the size of each edge in $H$ is in some set of positive integers $K$, then we say that $H$ is a \emph{$K$-hypergraph} and, if $H$ is a linear space, that $H$ is a \emph{linear $K$-space}. We abbreviate this to \emph{$k$-hypergraph} or \emph{linear $k$-space} when $K=\{k\}$ for some positive integer $k$. For a positive integer $k$, we denote by $H_k$ the $k$-hypergraph with vertex set $V(H)$ and edge multiset $\{E \in E(H):|E|=k\}$. For a vertex $x$ of $H$, $\deg_H(x)$ denotes the number of edges of $H$ that contain $x$. The \emph{total degree} of $H$ is $\sum_{x \in V(H)}\deg_H(x)=\sum_{E\in E(H)}|E|$. A \emph{graph} is a 2-hypergraph whose edge multiset is in fact a set. The \emph{underlying graph} $G$ of $H$ is the graph $G$ with vertex set $V(H)$ in which two vertices are adjacent if and only if they appear together in an edge of $H$. Similarly, the \emph{defect} $D$ of $H$ is the graph $D$ with vertex set $V(H)$ in which two vertices are adjacent if and only if they do not appear together in an edge of $H$. Note that a linear space on $m$ vertices has a defect which is an empty graph on $m$ vertices and has an underlying graph which is a complete graph on $m$ vertices. The complete graph on $m$ vertices is denoted by $K_{m}$. We say a graph is an \emph{even graph} if each of its vertices has even degree and say it is an \emph{odd graph} if each of its vertices has odd degree.

\begin{example}\label{X:linSpace}
Consider the hypergraph $H$ with vertex set $\{1,\ldots,8\}$ and edge set
\[\bigl\{\{1,2,3,4\},\{2,5,6\},\{2,7,8\},\{3,5,7\},\{3,6,8\},\{4,5,8\},\{4,6,7\},\{1,5\},\{1,6\},\{1,7\},\{1,8\}\bigr\}.\]
It can be checked that each pair of vertices occurs in exactly one edge, so $H$ is in fact a linear $\{2,3,4\}$-space. Now $H$ has $8$ vertices, $11$ edges and total degree $30$ and hence its existence shows that $Z_{2,2}(8,11) \geq 30$.
\end{example}

Theorems~\ref{T:aboveSTS} and \ref{T:belowSTS} essentially claim that values of $Z_{2,2}(m,n)$ usually meet the floor of one of three upper bounds, and otherwise are 1 less than this. We will use the following shorthands for these three upper bounds throughout the paper.
\begin{align*}
  U^+ &= \tfrac{1}{4}m(m-1)+\tfrac{3}{2}n, \\
  U^{0\,} &= \tfrac{1}{14}m(3m-4)+\tfrac{12}{7}n, \\
  U^- &= \tfrac{1}{6}m(m-1)+2n.
\end{align*}
The bounds are always understood to be functions of $m$ and $n$ and the values of $m$ and $n$ will be clear from context. The bounds $U^+$ and $U^-$ are special cases of an upper bound due to Roman \cite{Rom} and $U^0$ is a special case of an upper bound established by the authors in \cite{CheHorMam}.

\begin{lemma}\label{L:basicUpperBounds}
For positive integers $m$ and $n$ we have
\[Z_{2,2}(m,n) \leq \min\left\{\left\lfloor U^+\right\rfloor, \left\lfloor U^- \right\rfloor, \left\lfloor U^0 \right\rfloor \right\}.\]
\end{lemma}

\begin{proof}
We have $Z_{2,2}(m,n) \leq \min\{\lfloor U^+ \rfloor, \lfloor U^- \rfloor\}$ by \cite[Theorem 1]{Rom} and $Z_{2,2}(m,n) \leq \lfloor U^0 \rfloor$ by \cite[Theorem 1.1]{CheHorMam}.
\end{proof}

We will make use of the following two simple lemmas. The proof of the first is obvious.

\begin{lemma}\label{L:parity}
Let $H$ be a linear $3$-hypergraph of order $m$. The underlying graph $G$ of $H$ is an even graph and has $|E(G)| \equiv 0 \mod{3}$. Consequently, the defect $D$ of $H$ is an even graph if $m$ is odd, is an odd graph if $m$ is even, and has $|E(D)| \equiv \binom{m}{2} \mod{3}$.
\end{lemma}

\begin{lemma}\label{L:evenExistence}
Let $G$ be an even graph for which $E(G)=E(G_1) \cup \cdots \cup E(G_t)$ where $G_1,\ldots,G_t$ are pairwise edge-disjoint complete graphs such that
\begin{itemize}[itemsep=0mm,parsep=0mm]
    \item
$|V(G_i)|$ is even for each $i \in \{1,\ldots,t\}$; and
    \item
$|V(G_1)| \leq \cdots \leq |V(G_t)|$.
\end{itemize}
Then $t \geq |V(G_t)|+1$. Furthermore, if $t = |V(G_t)|+1$, then each vertex in $G_t$ is in precisely one of $G_1,\ldots,G_{t-1}$.
\end{lemma}

\begin{proof}
Because $G$ is an even graph and $|V(G_i)|$ is even for each $i \in \{1,\ldots,t\}$, each vertex in $V(G_t)$ must also be in $V(G_i)$ for some $i \in \{1,\ldots,t-1\}$. Because $G_1,\ldots,G_t$ are edge-disjoint, $V(G_i)$ contains at most one vertex in $V(G_t)$ for each $i \in \{1,\ldots,t-1\}$. Thus $t-1 \geq |V(G_t)|$ with equality only if each vertex in $G_t$ is in precisely one of $G_1,\ldots,G_{t-1}$.
\end{proof}


The next lemma establishes that $Z_{2,2}(m,n)$ cannot achieve $\lfloor U^- \rfloor$ in the cases where Theorem~\ref{T:belowSTS} specifies this.

\begin{lemma}\label{L:strongUpperBounds}
Let $m$ and $n$ be positive integers with $m$ odd and $n = \left\lfloor\tfrac{1}{3}\binom{m}{2}\right\rfloor-x$ for some $x \in S$, where $S=\{1,2,3,4\}$ if $m \equiv 1,3 \mod{6}$ and $S=\{0,1,3\}$ if $m \equiv 5 \mod{6}$. Then
\[Z_{2,2}(m,n) \leq  \left\lfloor U^- \right\rfloor-1.\]
\end{lemma}

\begin{proof}
Let $H$ be a linear hypergraph with $m$ vertices, $n$ edges and total degree $z$. In view of Lemma~\ref{L:basicUpperBounds} it suffices to suppose that $z=\lfloor U^- \rfloor$ and derive a contradiction. Let $d$ be the number of edges in the defect of $H$ and, for each $i \in \{1,\ldots,m\}$, let $n_i$ be the number of edges of $H$ of size $i$. Because $H$ has $n$ edges and has a defect with $d$ edges, we have
\begin{align}
\medop\sum_{i=1}^m n_i &= n \label{E:con1}\\
\medop\sum_{i=1}^m\tbinom{i}{2}n_i &= \tbinom{m}{2}-d\,. \label{E:con2}
\end{align}
By adding $2$ times \eqref{E:con1} to $\frac{1}{3}$ times \eqref{E:con2}, and then manipulating using $z=\sum_{i=1}^m in_i$, we obtain
\begin{equation}\label{E:weightedSumDef}
z= U^- - \tfrac{d}{3} - \medop\sum_{i=1}^{m} c_in_i
\end{equation}
where $c_i=2+\tfrac{1}{3}\tbinom{i}{2}-i$ for each $i \in \{1,\ldots,m\}$. Observe that $c_i$ is a convex quadratic in $i$ and that $c_i=0$ if $i\in\{3,4\}$, $c_i=\frac{1}{3}$ if $i \in \{2,5\}$ and $c_i \geq 1$ otherwise. Let $\delta=1$ if $m \equiv 5 \mod{6}$ and $\delta=0$ otherwise. By subtracting 3 times \eqref{E:con1} from \eqref{E:con2}, substituting $n=\left\lfloor\tfrac{1}{3}\binom{m}{2}\right\rfloor-x$ and noting that $\binom{m}{2}-3\left\lfloor\tfrac{1}{3}\binom{m}{2}\right\rfloor=\delta$ we have
\begin{equation}\label{E:edgeSizeRestr}
\medop\sum_{i=1}^m\bigl(\tbinom{i}{2}-3\bigr)n_i = 3x+\delta-d.
\end{equation}
Recall that $H_3$ is the $3$-hypergraph with vertex set $V(H)$ and edge multiset $\{E \in E(H):|E|=3\}$. Let $D_3$ be the defect of $H_3$ and note that $D_3$ is an even graph by Lemma~\ref{L:parity}. We will derive a contradiction to this fact in each case. \smallskip

\noindent \textbf{Case 1.} Suppose that $m \equiv 1,3 \mod{6}$ and hence $x \in \{1,2,3,4\}$, $\delta=0$, and $\lfloor U^- \rfloor = U^-$. Then, by \eqref{E:weightedSumDef} and the observation following it, $z=\lfloor U^- \rfloor$ if and only if $d=0$ and $n_i = 0$ for each $i \in \{1,\ldots,m\} \setminus \{3,4\}$. Solving \eqref{E:edgeSizeRestr} in this case reveals that $n_4=x$. But if $n_4=x$ and $d=0$, then $D_3$ is an edge-disjoint union of $x \leq 4$ copies of $K_4$ and (possibly) isolated vertices, and Lemma~\ref{L:evenExistence} gives the contradiction that $D_3$ is not an even graph.\smallskip

\noindent \textbf{Case 2.} Suppose that $m \equiv 5 \mod{6}$ and hence $x \in \{0,1,3\}$, $\delta=1$, and $\lfloor U^- \rfloor = U^--\frac{1}{3}$. Then, by \eqref{E:weightedSumDef} and the observation following it, $z=\lfloor U^- \rfloor$ if and only if one of the following holds.
\begin{itemize}[nosep]
    \item[(i)]
$d=0$, $n_2 = 1$ and $n_i = 0$ for each $i \in \{1,\ldots,m\} \setminus \{2,3,4\}$.
    \item[(ii)]
$d=0$, $n_5 = 1$ and $n_i = 0$ for each $i \in \{1,\ldots,m\} \setminus \{3,4,5\}$.
    \item[(iii)]
$d=1$ and $n_i = 0$ for each $i \in \{1,\ldots,m\} \setminus \{3,4\}$.
\end{itemize}
In subcase (i), solving \eqref{E:edgeSizeRestr} reveals that $n_4=x+1$. Then $D_3$ would be the edge-disjoint union of $x+1$ copies of $K_4$, a copy of $K_2$ and (possibly) isolated vertices. So Lemma~\ref{L:evenExistence} shows that $x=3$ and, furthermore, that each copy of $K_4$ must have a vertex that is in the copy of $K_2$ but is not in any other copy of $K_4$. This is clearly impossible as the copy of $K_2$ has only two vertices. In subcase (ii), solving \eqref{E:edgeSizeRestr} reveals that $n_4=x-2$ and hence that $x=3$. But then $D_3$ would be the edge-disjoint union of a copy of $K_4$, a copy of $K_5$ and isolated vertices, and so obviously is not an even graph. In subcase (iii), solving \eqref{E:edgeSizeRestr} reveals that $n_4=x$. But then $D_3$ would be the edge-disjoint union of one copy of $K_2$, $x \leq 3$ copies of $K_4$ and (possibly) isolated vertices, and Lemma~\ref{L:evenExistence} establishes that $D_3$ is not an even graph.
\end{proof}

Finally in this section, we show that $Z_{2,2}(m,n)$ cannot achieve $\lfloor U^+ \rfloor$ in the cases where Theorem~\ref{T:aboveSTS} specifies this.

\begin{lemma}\label{L:singleStrongUpperBound}
Let $m$ and $n$ be positive integers with $m \equiv 5 \mod{6}$ and $n = \left\lceil\tfrac{1}{3}\binom{m}{2}\right\rceil$. Then
\[Z_{2,2}(m,n) \leq  \left\lfloor U^+ \right\rfloor-1.\]
\end{lemma}

\begin{proof}
Let $H$ be a linear hypergraph with $m$ vertices, $n$ edges and total degree $z$. Note that $n = \frac{m(m-1)+4}{6}$ since $m \equiv 5 \mod{6}$ and hence $\left\lfloor U^+\right\rfloor = U^+ = \binom{m}{2}+1$. In view of Lemma~\ref{L:basicUpperBounds} it suffices to suppose that $z=U^+$ and derive a contradiction. Let $d$ be the number of edges in the defect of $H$ and, for each $i \in \{1,\ldots,n\}$, let $n_i$ be the number of edges of $H$ of size $i$. Note that \eqref{E:con1} and \eqref{E:con2} hold and by
adding $\frac{3}{2}$ times \eqref{E:con1} to $\frac{1}{2}$ times \eqref{E:con2}, and then manipulating using $z=\sum_{i=1}^m in_i$, we obtain
\[
z= U^+ - \tfrac{d}{2} - \medop\sum_{i=1}^{m} c_i n_i
\]
where $c_i=\tfrac{1}{2}\tbinom{i}{2}+\tfrac{3}{2}-i$. Thus, since $z=U^+$, we have $d=0$ and $\sum_{i=1}^{m} c_i n_i=0$. By noting that $c_i$ is a convex quadratic in $i$ with $c_i = 0$ if $i\in\{2,3\}$ and $c_i >0$ otherwise, it follows that $n_i=0$ for each $i \in \{1,\ldots,m\}\setminus\{2,3\}$. Thus, from \eqref{E:con1} and \eqref{E:con2} we can deduce $n_2=\frac{3}{2}n-\frac{1}{2}\binom{m}{2}=1$. So the defect $D_3$ of $H_3$ would have exactly one edge in contradiction to Lemma~\ref{L:parity}.
\end{proof}

In view of Lemmas~\ref{L:basicUpperBounds}, \ref{L:strongUpperBounds} and \ref{L:singleStrongUpperBound}, in order to prove Theorems~\ref{T:aboveSTS} and \ref{T:belowSTS}, it suffices to find linear hypergraphs with the appropriate total degrees and numbers of vertices and edges. We will accomplish this in Sections~\ref{S:above} and \ref{S:below}.

\section{Values of \texorpdfstring{$\bm{n}$}{n} above the triple system threshold}\label{S:above}

Our main goal in this section is to prove Theorem~\ref{T:aboveSTS}. Table~\ref{Tab:aboveSTS} details how we will divide the proof. It gives a division into cases based on the values of $m$ and $n$ and, for each case, the value of $Z_{2,2}(m,n)$ claimed by Theorem~\ref{T:aboveSTS} and the result that will establish this is indeed the case. 

\begin{table}[H]
\begin{center}
\begin{tabular}{ll|c|c}
    & case & $Z_{2,2}(m,n)$ & result \\
    \hline\hline
    $m$ even & $n \leq \lfloor\frac{1}{3}\binom{m}{2}+\frac{m}{3}\rfloor$ &  $\lfloor U^0 \rfloor$ & Lemma~\ref{L:evenZValue} \\
    $m\equiv 5 \mod{6}$ & $n=\lceil\frac{1}{3}\binom{m}{2}\rceil$ &  $\lfloor U^+ \rfloor-1$ & Lemma~\ref{L:onlyExceptionAbove} \\
    all other cases &  &  $\lfloor U^+ \rfloor$ & Lemma~\ref{L:con from pack}
\end{tabular}
\caption{Division of the proof of Theorem~\ref{T:aboveSTS}}\label{Tab:aboveSTS}
\end{center}
\end{table}

\subsection*{Achieving $\lfloor U^+ \rfloor$ or $\lfloor U^+ \rfloor-1$}

Recall that $U^+=\tfrac{1}{4}m(m-1)+\frac{3}{2}n$. To show that $Z_{2,2}(m,n)$ achieves $\lfloor U^+ \rfloor$ or $\lfloor U^+ \rfloor-1$ it is enough to use linear $\{2,3\}$-hypergraphs. Spencer \cite{Spencer}, determined the maximum number of edges in a partial Steiner triple system of a given order. His result immediately implies the following.

\begin{theorem}[\cite{Spencer}]\label{l:23-hypergraph}
For any integer $m$, there exists a linear $\{2,3\}$-space on $m$ vertices with $n$ edges of size $3$ if and only if either
\begin{itemize}[itemsep=0mm,parsep=0mm]
    \item
$m$ is even and $n \leq \lfloor\tfrac{1}{3}\binom{m}{2}-\frac{m}{6}\rfloor$;
    \item
$m \equiv 1,3 \mod{6}$ and $n \leq \lfloor\tfrac{1}{3}\binom{m}{2}\rfloor$; or
    \item
$m \equiv 5 \mod{6}$ and $n \leq \lfloor\tfrac{1}{3}\binom{m}{2}\rfloor-1$.
\end{itemize}
\end{theorem}

From Theorem~\ref{l:23-hypergraph} we can establish that $Z_{2,2}(m,n)=\lfloor U^+ \rfloor$ in the cases where Theorem~\ref{T:aboveSTS} specifies this. In fact, this was shown by Guy \cite{Guy} but we give a proof here for the sake of completeness.

\begin{lemma}\label{L:con from pack}
For positive integers $m$ and $n$ with $ n \leq \binom{m}{2}$,
$Z_{2,2}(m,n) = \lfloor U^+\rfloor$
if either
\begin{itemize}[itemsep=0mm,parsep=0mm]
    \item
$m$ is even and $n \geq \left\lceil\tfrac{1}{3}\binom{m}{2}+\frac{m}{3}\right\rceil$;
    \item
$m \equiv 1,3 \mod{6}$ and $n \geq \lceil\tfrac{1}{3}\binom{m}{2}\rceil$; or
    \item
$m \equiv 5 \mod{6}$ and $n \geq \lceil\tfrac{1}{3}\binom{m}{2}\rceil+1$.
\end{itemize}
\end{lemma}
\begin{proof}
In view of Lemma~\ref{L:basicUpperBounds}, it suffices to find a linear hypergraph with $m$ vertices, $n$ edges, and total degree $\lfloor U^+ \rfloor$. Let $n_3=\lfloor\frac{1}{2}\binom{m}{2}-\frac{n}{2}\rfloor$ and $n_2=n-n_3$. Note that $n_2$ and $n_3$ are nonnegative because $\frac{1}{3}\binom{m}{2} \leq n \leq \binom{m}{2}$. We will use Theorem~\ref{l:23-hypergraph} to construct a linear $\{2,3\}$-hypergraph with $n_2$ edges of size 2 and $n_3$ edges of size $3$. This will suffice to complete the proof because $n_2+n_3=n$ and $2n_2+3n_3=\lfloor U^+ \rfloor$.

Using the hypothesised bounds on $n$, it is routine to check that $n_3$ 
satisfies the hypotheses of Theorem~\ref{l:23-hypergraph}. Thus there is a linear $\{2,3\}$-space $H$ on $m$ vertices with $|E(H_3)|=n_3$ and hence with $|E(H_2)|=\binom{m}{2}-3n_3$. We can obtain the desired hypergraph from this linear space by deleting edges of size 2 provided that $\binom{m}{2}-3n_3$ is at least $n_2=n-n_3$. This is the case because
\[\tbinom{m}{2}-3n_3-n_2=\tbinom{m}{2}-n-2\left\lfloor\tfrac{1}{2}\tbinom{m}{2}-\tfrac{n}{2}\right\rfloor \geq 0. \qedhere\]
\end{proof}

Theorem~\ref{l:23-hypergraph} also allows us to resolve the special case of Theorem~\ref{T:aboveSTS} when $m \equiv 5 \mod{6}$ and $n = \lceil \frac{1}{3}\binom{m}{2}\rceil$.

\begin{lemma}\label{L:onlyExceptionAbove}
If $m \equiv 5 \mod{6}$ and $n = \lceil \frac{1}{3}\binom{m}{2}\rceil$, then $Z_{2,2}(m,n)=\lfloor U^+ \rfloor-1$.
\end{lemma}

\begin{proof}
Note that $n=\lceil\frac{1}{3}\binom{m}{2}\rceil=\frac{1}{3}\binom{m}{2}+\frac{2}{3}$ since $m\equiv 5 \pmod{6}$. By applying Theorem~\ref{l:23-hypergraph} with $n_3=\frac{1}{3}\binom{m}{2} - \frac{4}{3}$ and deleting two of the four edges of size 2 from the resulting linear space, there exists a linear $\{2,3\}$-hypergraph $H$ with $|E(H_2)|=2$ and $|E(H_3)|=\frac{1}{3}\binom{m}{2} - \frac{4}{3}$. It is easy to check that $H$ has $\lceil \frac{1}{3}\binom{m}{2}\rceil$ edges and total degree $\lfloor U^+ \rfloor -1=\binom{m}{2}$. Thus Lemma~\ref{L:singleStrongUpperBound} completes the proof.
\end{proof}

\subsection*{Achieving $\lfloor U^0 \rfloor$}


Recall that $U^{0} = \tfrac{1}{14}m(3m-4)+\tfrac{12}{7}n$. In general it does not suffice to use $\{2,3\}$-hypergraphs and $\{3,4\}$-hypergraphs to establish that $Z_{2,2}(m,n)$ meets $\lfloor U^0 \rfloor$ in the cases where Theorems~\ref{T:aboveSTS} and \ref{T:belowSTS} specify this. Instead we require $\{2,3,4\}$-hypergraphs. To construct these, we make use of results on group divisible designs. Let $g_1,\ldots,g_s$ be distinct positive integers and $a_1,\ldots,a_s$ be nonnegative integers. A \emph{$3$-$\GDD$ of type $g_1^{a_1}g_2^{a_2}\ldots g_s^{a_s}$} is a linear $\{3,g_1,\ldots,g_s\}$-space with $\sum_{i=1}^sa_ig_i$ vertices whose edge multiset admits a partition $\{\B,\mathcal{G}\}$ such that each edge in $\B$ has size $3$, $\mathcal{G}$ contains exactly $a_i$ edges of size $g_i$ for each $i \in \{1,\ldots,s\}$, and the edges in $\mathcal{G}$ form a partition of the vertex set. The edges in $\mathcal{G}$ are called the \emph{groups} of the $3$-$\GDD$.

We will make use of the following two results on the existence of $3$-GDDs. \pagebreak

\begin{theorem}[\cite{Zhu}]\label{3-GDD-uni}
Let $h$ and $w$ be positive integers. There exists a $3$-$\GDD$ of type $h^w$ if and only if
\begin{itemize}[itemsep=0mm,parsep=0mm,topsep=1mm]
    \item[\textup{(i)}]
$w\geq 3$ or $w=1$;
    \item[\textup{(ii)}]
$(w-1)h\equiv 0 \pmod{2}$; and
    \item[\textup{(iii)}]
$w(w-1)h^2 \equiv 0 \mod{6}$.
\end{itemize}
\end{theorem}

\begin{theorem}[\cite{Colbourn}]\label{3-GDD}
Let $g$, $h$ and $w$ be positive integers with $g \neq h$. There exists a $3$-$\GDD$ of type $g^1h^w$ if and only if
\begin{itemize}[itemsep=0mm,parsep=0mm,topsep=1mm]
    \item[\textup{(i)}]
$w \geq 3$;
    \item[\textup{(ii)}]
$g\leq h(w-1)$;
    \item[\textup{(iii)}]
$h(w-1)+g\equiv 0 \mod{2}$;
    \item[\textup{(iv)}]
$hw\equiv0 \mod{2}$; and
    \item[\textup{(v)}]
$\frac{1}{2}h^2w(w-1)+ghw \equiv 0 \mod{3}$.
\end{itemize}
\end{theorem}

We also note that, in \cite{Colbourn0}, the existence problem is solved for every possible type of $3$-$\GDD$ with at most 60 vertices. In addition we will make use of Theorem~\ref{T:3GDD4s2s}, which we now prove.

\begin{proof}[\textbf{\textup{Proof of Theorem~\ref{T:3GDD4s2s}.}}]
If a $3$-$\GDD$ of type $4^u2^v$ exists, then $16\binom{u}{2}+4\binom{v}{2}+8uv \equiv 0 \mod{3}$ by counting pairs of vertices in edges of size 3. Hence, multiplying through by 2, we must have $u(u-1)+v(v-1)+uv \equiv 0 \mod{3}$. From this it is easy to deduce that (i) and (ii) hold. So it remains to show that (i) and (ii) imply the existence of a $3$-$\GDD$ of type $4^u2^v$. If $u=0$ or $v=0$, then this can be seen to follow from Theorem~\ref{3-GDD-uni}. If $u=1$ or $v=1$, this can be seen to follow from Theorem~\ref{3-GDD}. So we may assume $u \geq 3$ and $v \geq 3$. We divide the proof into two cases according to the congruence class of $u$ modulo 3.\smallskip

\noindent\textbf{Case 1.} Suppose that $u \equiv 0 \mod{3}$. We further divide this case according to the value of $v$, remembering that $v \equiv 0,1 \mod{3}$.\smallskip

\noindent\textbf{Case 1a.} Suppose that $v \leq 2u-2$. Then there exists a $3$-$\GDD$ of group type $(2v)^14^u$, by Theorem~\ref{3-GDD}. The required $3$-$\GDD$ can be obtained from this by replacing the group of size $2v$ with a $3$-$\GDD$ of type $2^v$, which exists by Theorem~\ref{3-GDD-uni}.

\noindent\textbf{Case 1b.} Suppose that $v = 2u$. Since $u\equiv 0 \pmod 3$, it follows
that $v  \equiv 0 \mod {6}$. It was shown in \cite{Colbourn0} that a $3$-$\GDD$ of type $4^32^6$ exists, so we may assume $u \geq 6$. Then $\frac{2u}{3} \geq 4$, so there exists a $3$-$\GDD$ of type $12^{2u/3}$ by Theorem~\ref{3-GDD-uni}. The required $3$-$\GDD$ can be obtained from this by replacing $\frac{u}{3}$ groups with $3$-GDDs of type $4^3$ and the remaining $\frac{u}{3}=\frac{v}{6}$ groups with $3$-GDDs of type $2^6$. Both of the ingredients exist by Theorem~\ref{3-GDD-uni}.

\noindent\textbf{Case 1c.} Suppose that $v \geq 2u+1$. Then there exists a $3$-$\GDD$ of type  $(4u)^12^{v}$ by Theorem~\ref{3-GDD}.
The required $3$-$\GDD$ can be obtained from this by replacing the group of size $4u$ with a $3$-$\GDD$ of group type $4^u$, which exists by Theorem~\ref{3-GDD-uni}.

\noindent\textbf{Case 2.} Suppose that $u \equiv 1 \mod{3}$ and hence that $v \equiv 0 \mod{3}$. If $4u+2v < 40$, then $(u,v) \in \{(4,3),(4,6),(4,9),(7,3)\}$ and, in each of these cases, a $3$-$\GDD$ of type $4^u2^v$ has been shown to exist in \cite{Colbourn0}. Thus we may suppose that $4u+2v \geq 40$. Let $\ell=4$ if $v \equiv 0 \mod{6}$ and $\ell=10$ if $v \equiv 3 \mod{6}$. Let $w=\frac{4u+2v-\ell}{12}$ and note that $w$ is an integer and $w \geq 3$ using $4u+2v \geq 40$ and the definition of $\ell$. By Theorem~\ref{3-GDD}, there is a $3$-$\GDD$ of type $\ell^1 12^w$. The required $3$-$\GDD$ can be obtained from this by replacing $\frac{u-1}{3}$ groups of size 12 with $3$-GDDs of type $4^3$, the remaining $\frac{2v+4-\ell}{12}$ groups of size 12 with $3$-GDDs of type $2^6$ and, if $\ell=10$, the group of size 10 with a $3$-GDD of type $4^1 2^3$. The former two ingredients exist by Theorem~\ref{3-GDD-uni} and the last exists by Theorem~\ref{3-GDD}.
\end{proof}

Using Theorems~\ref{T:3GDD4s2s} and \ref{3-GDD} we can prove the following lemma which will supply the linear $\{2,3,4\}$-hypergraphs we require.

\begin{lemma}\label{L:evenMHypergraphs}
Let $m$ and $n_4$ be integers such that $m \geq 2$ is even, $0 \leq n_4 \leq \lfloor \frac{m}{4} \rfloor$ and $n_4 \neq \lfloor \frac{m}{4} \rfloor$ if $m \equiv 6, 8, 10 \mod{12}$. There exists a linear $\{2,3,4\}$-space on $m$ vertices with $n_4$ edges of size $4$ and $\lfloor\frac{m(m-2)}{6}-\frac{4}{3}n_4\rfloor$ edges of size $3$.
\end{lemma}

\begin{proof}
A routine calculation reveals that a linear $\{2,3,4\}$-space with $n_4$ edges of size 4 and $\frac{m}{2}-2n_4+j$ edges of size 2 for some $j \in \{0,1,2\}$ will necessarily have the desired number of edges of size 3. We first deal with the special case  where $m\equiv 8 \mod{12}$ and $n_4=\frac{m}{4}-1$. In this case, the desired linear space can be obtained by taking a 3-$\GDD$ of type $8^1 4^{m/4-2}$ and replacing the group of size 8 with a $\{2,3,4\}$-linear space on 8 vertices with one edge of size 4 and four edges of size 2. The first of these ingredients is trivial if $m=8$ and exists by Theorem~\ref{3-GDD} otherwise, while the second exists by Example~\ref{X:linSpace}. The resulting linear space has four edges of size 2 and hence the required number of edges of size 3. So, in view of our hypotheses, we may now assume that $n_4 \leq \lfloor \frac{m}{4} \rfloor -1$ if $m \equiv 6, 10 \mod{12}$ and $n_4 \leq \lfloor \frac{m}{4} \rfloor -2$ if $m \equiv 8 \mod{12}$.

Let $a$ be the least integer such that $a \geq n_4$, $a \equiv 1 \mod{3}$ if $m \equiv 4 \mod{6}$ and $a \equiv 0 \mod{3}$ otherwise. It can be checked that $a \leq \lfloor \frac{m}{4} \rfloor$ by considering cases according to the value of $m$ modulo 12, deducing the corresponding values of $\lfloor \frac{m}{4} \rfloor$ modulo 3, and bearing in mind our assumed upper bounds on $n_4$. There is a 3-$\GDD$ with group type $4^{a} 2^{m/2-2a}$ by Theorem~\ref{T:3GDD4s2s}. Let $i=a-n_4$, note $i \in \{0,1,2\}$ and replace each of $i$ edges of size 4 in this 3-$\GDD$ with a linear $\{2,3\}$-space on 4 vertices with one edge of size 3. The resulting linear $\{2,3,4\}$-space has $n_4$ edges of size 4 and $\frac{m}{2}-2a+3i=\frac{m}{2}-2n_4+i$ edges of size 2, and hence has the desired number of edges of size 3.
\end{proof}

Using Lemma~\ref{L:evenMHypergraphs} we can now establish that $Z_{2,2}(m,n)=\lfloor U^0 \rfloor$ in all of the cases where Theorem~\ref{T:aboveSTS} specifies this and, further, in almost all of the cases where Theorem~\ref{T:belowSTS} specifies this.

\begin{lemma}\label{L:evenZValue}
We have $Z_{2,2}(m,n) = \lfloor U^0 \rfloor$ for all positive integers $m$ and $n$ such that $m$ is even, $\lfloor\frac{1}{3}\binom{m}{2}-\frac{m}{4}\rfloor+1 \leq n \leq \lfloor \frac{1}{3}\binom{m}{2}+\frac{m}{3} \rfloor$ and $n \neq \lceil \frac{1}{3}\binom{m}{2}-\frac{m}{4} \rceil$ if $m \equiv 6, 8, 10 \mod{12}$.
\end{lemma}

\begin{proof}
In view of Lemma~\ref{L:basicUpperBounds}, it suffices to find a linear hypergraph with $m$ vertices, $n$ edges, and total degree $\lfloor U^0 \rfloor$. We claim there exists a linear $\{2,3,4\}$-hypergraph $H$ with $n_i$ edges of size $i$ for each $i \in \{2,3,4\}$, where $n_2$, $n_3$ and $n_4$ are defined as follows. Let $n^*_4=\lfloor\frac{1}{14}m(m+1) - \frac{3}{7}n  \rfloor$. We usually take $n_4=n_4^*$, except that we take $n_4=n_4^*+1$ in the special cases when
\begin{equation}\label{E:specialCaseCond}
\tfrac{1}{2}m(m-2)-4n^*_4 \equiv 2 \mod{3} \quad\text{ and }\quad \tfrac{1}{2}m(m+1) - 3n \equiv \text{4, 5 or 6} \mod{7}
\end{equation}
Let $n_3=\lfloor\frac{m(m-2)}{6}-\frac{4}{3}n_4\rfloor$ and $n_2=n-n_3-n_4$.

It will be convenient to further define $\epsilon_3$ and $\epsilon_4$ to be the rational numbers such that $n_3=\frac{m(m-2)}{6}-\frac{4}{3}n_4-\epsilon_3$ and $n_4=\frac{1}{14}m(m+1)-\frac{3}{7}n-\epsilon_4$. Bearing in mind our definition of $n_4$, we have the following.
\begin{itemize}[itemsep=0mm,parsep=0mm]
    \item[(i)]
$\epsilon_3 \in \{0,\frac{1}{3},\frac{2}{3}\}$.
    \item[(ii)]
If \eqref{E:specialCaseCond} holds, then $\epsilon_3 = \frac{2}{3}$ and $\epsilon_4 \in \{-\frac{3}{7},-\frac{2}{7},-\frac{1}{7}\}$.
    \item[(iii)]
If \eqref{E:specialCaseCond} does not hold, then $\epsilon_4 \in \{0,\frac{1}{7},\ldots,\frac{6}{7}\}$ and, further, $\epsilon_4 \in \{0,\frac{1}{7},\frac{2}{7},\frac{3}{7}\}$ if  $\epsilon_3 = \frac{2}{3}$.
\end{itemize}

The total degree of a hypergraph with $n_i$ edges of size $i$ for each $i \in \{2,3,4\}$, is
\begin{equation}\label{E:degCount}
2n_2+3n_3+4n_4=2n_4+n_3+2n = \tfrac{2}{3}n_4+\tfrac{m(m-2)}{6}-\epsilon_3+2n  = U^0-\tfrac{2}{3}\epsilon_4-\epsilon_3
\end{equation}
where the equalities follows using, in turn, the definitions of $n_2$, $\epsilon_3$ and $\epsilon_4$. Using (i), (ii) and (iii) it can be checked that the final expression in \eqref{E:degCount} is greater than $U^0-1$ and hence is at least $\lfloor U^0 \rfloor$ since the first expression in \eqref{E:degCount} is clearly an integer. Thus it will indeed suffice to show that a hypergraph $H$ with the specified properties exists.

We now show that $n_4$ satisfies the hypotheses of Lemma~\ref{L:evenMHypergraphs}. Let $n'$ be the rational number such that $n=\frac{1}{3}\binom{m}{2}-\frac{m}{4}+n'$ and note that, by our hypotheses $0 < n' \leq \frac{7m}{12}$. Using $n=\frac{1}{3}\binom{m}{2}-\frac{m}{4}+n'$ and the definition of $\epsilon_4$, we have that
\[n_4 = \tfrac{m}{4}-\tfrac{3n'}{7}-\epsilon_4.\]
In all cases this implies that $0 \leq n_4 \leq \lfloor\frac{m}{4}\rfloor$, bearing in mind that $n_4$ is an integer, that $0 < n' \leq \frac{7m}{12}$, that $m$ is even and that $-\frac{3}{7} \leq \epsilon_4 \leq \frac{6}{7}$ by (ii) and (iii). When $m \equiv 8 \mod{12}$ we have $n_4 \neq \frac{m}{4}$ since $n' \geq \frac{5}{3}$. When $m \equiv 6, 10 \mod{12}$, we have $n_4 \neq \frac{m-2}{4}$ because $n' = \frac{3}{2}+j$ for some nonnegative integer $j$ and, if $j = 0$, then $n^*_4=\lfloor\frac{m}{4}-\frac{9}{14}\rfloor=\frac{m-6}{4}$, so the first condition of \eqref{E:specialCaseCond} fails and $n=n^*_4$.

So by Lemma~\ref{L:evenMHypergraphs} there is a linear $\{2,3,4\}$-space with $n_i$ edges of size $i$ for each $i \in \{3,4\}$. We can obtain the required hypergraph $H$ from this space by (if needed) deleting edges of size 2 provided that $0 \leq n_2 \leq \binom{m}{2}-3n_3-6n_4$. So it suffices to prove these two inequalities.

To see that $n_2$ is nonnegative note that
\begin{equation}\label{E:n2Pos}n_2=n-n_3-n_4= 
\tfrac{6}{7}n -\tfrac{1}{14}m(2m-5)+\epsilon_3 - \tfrac{1}{3}\epsilon_4 > \epsilon_3-\tfrac{1}{3}\epsilon_4
\end{equation}
where second equality follows using, successively, the definitions of $\epsilon_3$ and $\epsilon_4$, and the inequality follows using $n>\frac{1}{3}\binom{m}{2}-\frac{m}{4}$ (which in turn follows from our hypotheses). Using (i), (ii) and (iii) it can be checked that the final expression in \eqref{E:n2Pos} is at least $-1$. Thus $n_2$ is nonnegative since it is an integer by definition. To see that $n_2 \leq \binom{m}{2}-3n_3-6n_4$, 
note that
\begin{equation}\label{E:edgeCount}
n_2+3n_3+6n_4 = n+2n_3+5n_4=
\tbinom{m}{2}-2\epsilon_3-\tfrac{7}{3}\epsilon_4.
\end{equation}
where the equalities follow using, successively, the definitions of $n_2$, $\epsilon_3$ and $\epsilon_4$. Using (i), (ii) and (iii) it can be checked that the final expression in \eqref{E:edgeCount} is less than $\binom{m}{2}+1$ and hence is at most $\binom{m}{2}$ since the first expression in \eqref{E:edgeCount} is an integer. This completes the proof.
\end{proof}

\subsection*{Proof of Theorem~\ref{T:aboveSTS}}

We can now prove Theorem~\ref{T:aboveSTS} by combining our previous results.

\begin{proof}[\textup{\textbf{Proof of Theorem~\ref{T:aboveSTS}.}}]
Each pair $(m,n)$ of parameters covered by Theorem~\ref{T:aboveSTS} falls into exactly one of the cases given in the first column of Table~\ref{Tab:aboveSTS}. For each case, Theorem~\ref{T:aboveSTS} claims that the value of $Z_{2,2}(m,n)$ is as given in the second column, and the result listed in the third column establishes that this is indeed the case.
\end{proof}

\section{Values of \texorpdfstring{$\bm{n}$}{n} below the triple system threshold}\label{S:below}

In this section we will prove Theorem~\ref{T:belowSTS}. Table~\ref{Tab:belowSTS} details how we will divide the proof. It gives a division into cases based on the values of $m$ and $n$ and, for each case, the value of $Z_{2,2}(m,n)$ claimed by Theorem~\ref{T:belowSTS} and the result that will establish this value.

\begin{table}[H]
\begin{center}
\begin{tabular}{ll|l|l}
    \multicolumn{2}{c|}{case} & $Z_{2,2}(m,n)$ & result \\
    \hline\hline
    $m\equiv 1,3 \mod{6}$ & $n=\lfloor\frac{1}{3}\binom{m}{2}\rfloor-\{1,2,3,4\}$ &  $\lfloor U^- \rfloor-1$ & Lemma~\ref{L:not2Mod3ExceptionalZValues} \\
    $m\equiv 1,3 \mod{6}$ & $n \leq \lfloor\frac{1}{3}\binom{m}{2}\rfloor-5$ &  $\lfloor U^- \rfloor$ & Lemma~\ref{L:ColRosStiConsequence} \\
    $m\equiv 5 \mod{6}$ & $n=\lfloor\frac{1}{3}\binom{m}{2}\rfloor-\{0,1,3\}$ &  $\lfloor U^- \rfloor-1$ & Lemma~\ref{L:5mod6ZValue} \\
    $m\equiv 5 \mod{6}$ & $n=\lfloor\frac{1}{3}\binom{m}{2}\rfloor-\{2\} \cup \{4,5,\ldots,27\}$ &  $\lfloor U^- \rfloor$ & Lemma~\ref{L:5mod6ZValue} \\
    $m\equiv 5 \mod{6}$ & $n \leq \lfloor\frac{1}{3}\binom{m}{2}\rfloor-28$ &  $\lfloor U^- \rfloor$ & Lemma~\ref{T:CHMexact} \\ \hline
    $m$ even & $n > \lfloor\frac{1}{3}\binom{m}{2}-\frac{m}{4}\rfloor$ &  $\lfloor U^0 \rfloor$ & Lemmas~\ref{L:evenZValue},~\ref{L:U0Exceptional} \\
    $m \equiv 0,4 \mod{6}$ & $n \leq \lfloor\frac{1}{3}\binom{m}{2}-\frac{m}{4}\rfloor$ &  $\lfloor U^- \rfloor$ & Lemma~\ref{L:ColRosStiConsequence} \\
    $m \equiv 2 \mod{6}$ & $n=\lfloor\frac{1}{3}\binom{m}{2}-\frac{m}{4}\rfloor-\{0,\ldots,27\}$ &  $\lfloor U^- \rfloor$ & Lemma~\ref{L:2mod6ZValue} \\
    $m \equiv 2 \mod{6}$ & $n \leq \lfloor\frac{1}{3}\binom{m}{2}-\frac{m}{4}\rfloor-28$ &  $\lfloor U^- \rfloor$ & Lemma~\ref{T:CHMexact} \\
\end{tabular}
\caption{Division of the proof of Theorem~\ref{T:belowSTS}}\label{Tab:belowSTS}
\end{center}
\end{table}

Results of Glock et al. \cite{GloKuhLoMonOst} and Delcourt and Postle \cite{DelPos} establish that any graph of sufficiently large order $m$ with minimum degree at least $\frac{1}{14}(7+\sqrt{21})m\approx 0.82733m$ has a decomposition into triangles if and only if it satisfies the obvious necessary conditions. This can be stated as follows.

\begin{theorem}[\cite{DelPos,GloKuhLoMonOst}]\label{L:3HypergraphExistence}
There is a positive integer $m_0$ such that if $G$ is an even graph of order $m > m_0$, minimum degree at least $\frac{1}{14}(7+\sqrt{21})m$ and $|E(G)| \equiv 0\mod{3}$, then there is a linear $3$-hypergraph with underlying graph $G$.
\end{theorem}

In fact, we do not require the full strength of Theorem~\ref{L:3HypergraphExistence}: $\frac{1}{14}(7+\sqrt{21})$ could be replaced with any constant less than 1 and the result would still suffice for our purposes.

\subsection*{Achieving $\bm{\lfloor U^0 \rfloor}$}

Recall that $U^{0} = \tfrac{1}{14}m(3m-4)+\tfrac{12}{7}n$. Lemma~\ref{L:evenZValue} covers the vast majority of the cases for which Theorem~\ref{T:belowSTS} claims $Z_{2,2}(m,n)=\lfloor U^0 \rfloor$. In the following lemma we deal with the remaining cases.

\begin{lemma}\label{L:U0Exceptional}
For each sufficiently large integer $m$ such that $m \equiv 6,8,10 \pmod{12}$ and for $n= \lceil \frac{1}{3}\binom{m}{2}-\frac{m}{4} \rceil$ we have $Z_{2,2}(m,n)=\lfloor U^0 \rfloor$.
\end{lemma}

\begin{proof}
In view of Lemma~\ref{L:basicUpperBounds}, it suffices to find a linear hypergraph with $m$ vertices, $n$ edges, and total degree $\lfloor U^0 \rfloor$. We will show there exists a linear $\{2,3,4\}$-hypergraph $H$ on $m$ vertices with $|E(H_2)|$, $|E(H_3)|$ and $|E(H_4)|$ as specified in the following table. In each case simple calculations show that such an $H$ would have $n$ edges and total degree $\lfloor U^0 \rfloor = \lfloor\tfrac{1}{14}m(3m-4)+\tfrac{12}{7}n\rfloor$, as required to complete the proof.

\begin{center}
\begin{tabular}{ll||c|c|c}
  \multicolumn{2}{c||}{case} & \rule{4mm}{0mm}$|E(H_2)|$\rule{4mm}{0mm} & \rule{4mm}{0mm}$|E(H_3)|$\rule{4mm}{0mm} & \rule{4mm}{0mm}$|E(H_4)|$\rule{4mm}{0mm} \\ 
  \hline
  $m \equiv 6, 10 \pmod{12}$ & $n=\frac{1}{3}\binom{m}{2}-\frac{m}{4}+\frac{1}{2}$ & $1$ & $\frac{1}{3}\binom{m}{2}-\frac{m}{2}$ & $\frac{m-2}{4}$ \\ 
  $m \equiv 8 \pmod{12}$ & $n=\frac{1}{3}\binom{m}{2}-\frac{m}{4}+\frac{2}{3}$ & $1$ & $\frac{1}{3}\binom{m}{2}-\frac{m}{2}-\frac{1}{3}$ & $\frac{m}{4}$ \\ 
\end{tabular}
\end{center}

To see that such a hypergraph exists consider the appropriate linear $\{2,4\}$-hypergraph $H'$ on $m$ vertices depicted in Figure~\ref{F:U0Exceptional}. It can be seen that the defect $D'$ of this hypergraph is an even graph with $|E(D')| \equiv 0 \mod {3}$ and each vertex having degree at least $m-8$. Thus, by Theorem~\ref{L:3HypergraphExistence}, there is a linear $3$-hypergraph $H''$ whose underlying graph is $D'$ provided that $m$ is sufficiently large.
Let $H$ be the linear hypergraph formed by taking the union of $H'$ and $H''$ and, if $m \equiv 6, 10 \pmod{12}$, deleting two edges of size 2. It can be checked that $|E(H_2)|$, $|E(H_3)|$ and $|E(H_4)|$ are as specified in the table.
\end{proof}

\begin{figure}[H]
\begin{center}
\includegraphics[width=\textwidth]{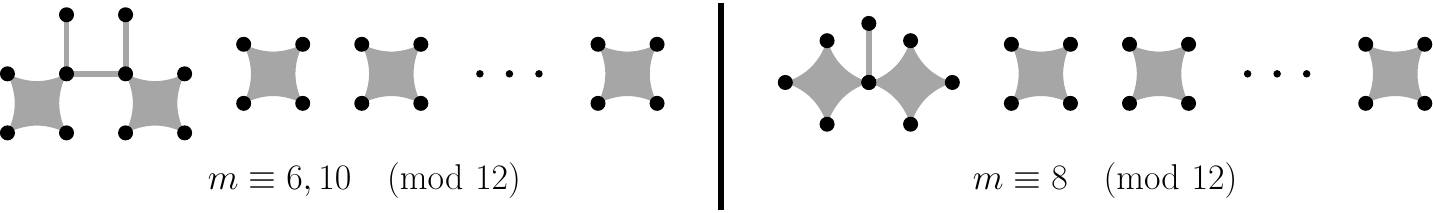}
\caption{The linear $\{2,4\}$-hypergraph $H'$ on $m$ vertices in the proof of Lemma~\ref{L:U0Exceptional}. Here and in subsequent figures, we depict edges as shaded regions that touch the vertices they are incident with.}\label{F:U0Exceptional}
\end{center}
\end{figure}


\subsection*{Achieving $\bm{\lfloor U^- \rfloor}$ or $\bm{\lfloor U^- \rfloor-1}$ when $\bm{m \not\equiv 2 \mod{3}}$}

Recall that $U^-=\tfrac{1}{6}m(m-1)+2n$. In cases where $m \not\equiv 2 \mod{3}$ we use a result due to Colbourn, Rosa and Stinson \cite{ColRosSti} on the existence of linear $\{3,4\}$-spaces. We state a slightly weaker version of it that will be sufficient for our purposes.

\begin{theorem}[\cite{ColRosSti}]\label{T:ColRosSti}
Let $m$ and $n_4$ be nonnegative integers such that $m \geq 96$, $m \not\equiv 2 \mod {3}$ and $n_4 \leq \lfloor \frac{1}{6}\binom{m}{2} - \frac{m}{6} \rfloor$. There exists a linear $\{3,4\}$-space on $m$ vertices with $n_4$ edges of size $4$ provided that $n_4 \geq 5$ if $m$ is odd and $n_4 \geq \lceil \frac{m}{4} \rceil$ if $m$ is even.
\end{theorem}

We can use this result to settle the cases of Theorem~\ref{T:belowSTS} where $m \not\equiv 2 \mod{3}$ and $Z_{2,2}(m,n)=\lfloor U^- \rfloor$.

\begin{lemma}\label{L:ColRosStiConsequence}
We have $Z_{2,2}(m,n)=\lfloor U^- \rfloor$ for all nonnegative integers $m$ and $n$ such that $m \geq 96$, $m \not\equiv 2 \mod {3}$, and
\[\tfrac{1}{6}\tbinom{m}{2}+\tfrac{m}{3}+40 \leq n \leq
\left\{
  \begin{array}{ll}
    \left\lfloor\tfrac{1}{3}\binom{m}{2}-\frac{m}{4}\right\rfloor & \hbox{if $m$ is even} \\[1mm]
    \left\lfloor\tfrac{1}{3}\binom{m}{2}\right\rfloor-5 & \hbox{if $m$ is odd.}
  \end{array}
\right.\]
\end{lemma}

\begin{proof}
Note that $\tfrac{1}{3}\binom{m}{2}$ is an integer since $m \not\equiv 2 \mod {3}$. In view of Lemma~\ref{L:basicUpperBounds}, it suffices to find a linear hypergraph with $m$ vertices, $n$ edges, and total degree $\lfloor U^- \rfloor$. Let $n_4=\frac{1}{3}\binom{m}{2}-n$ and $n_3=n-n_4$. We will use Theorem~\ref{T:ColRosSti} to construct a linear $\{3,4\}$-hypergraph $H$ with $n_4$ edges of size $4$ and $n_3$ edges of size 3. This will suffice to complete the proof because $n_3+n_4=n$ and $4n_4+3n_3=3n+n_4=\lfloor U^- \rfloor$.

Note that $n_3$ is nonnegative by our hypothesised lower bound on $n$. Furthermore, our bounds on $n$ imply that $n_4$ satisfies the hypotheses of Theorem~\ref{T:ColRosSti} and hence there is a linear $\{3,4\}$-space $H$ on $m$ vertices with $|E(H_4)|=n_4$ and hence with $|E(H_3)|=\frac{1}{3}\binom{m}{2}-2n_4$. Using the definitions of $n_3$ and $n_4$ we have that $\frac{1}{3}\binom{m}{2}-2n_4=n-n_4=n_3$ as required.
\end{proof}

The following lemma will deal with the cases of Theorem~\ref{T:belowSTS} where $m \not\equiv 2 \mod{3}$ and $Z_{2,2}(m,n)=\lfloor U^- \rfloor-1$.

\begin{lemma}\label{L:not2Mod3ExceptionalZValues}
For each sufficiently large integer $m$ such that $m \equiv 1,3 \pmod{6}$ and for $n=\frac{1}{3}\binom{m}{2}-r$ where $r \in \{1,2,3,4\}$, we have $Z_{2,2}(m,n) = \lfloor U^-\rfloor-1$.
\end{lemma}

\begin{proof}
In view of Lemma~\ref{L:strongUpperBounds}, it suffices to find a linear hypergraph with $m$ vertices, $n$ edges, and total degree $\lfloor U^- \rfloor-1$. We will show there exists a linear $\{2,3,4\}$-hypergraph $H$ on $m$ vertices with $|E(H_2)|$, $|E(H_3)|$ and $|E(H_4)|$ as specified in the following table. In each case simple calculations show that such an $H$ will have $n=\frac{1}{3}\binom{m}{2}-r$ edges and total degree $\lfloor U^-\rfloor-1 = \binom{m}{2} - 2r -1$, as required to complete the proof.

\begin{center}
\begin{tabular}{c||c|c|c}
  \rule{4mm}{0mm}$r$\rule{4mm}{0mm} & \rule{4mm}{0mm}$|E(H_2)|$\rule{4mm}{0mm} & \rule{4mm}{0mm}$|E(H_3)|$\rule{4mm}{0mm} & \rule{4mm}{0mm}$|E(H_4)|$\rule{4mm}{0mm} \\ 
  \hline
  1 & $0$ & $\frac{1}{3}\binom{m}{2}-1$ & $0$  \\ 
  2 & $1$ & $\frac{1}{3}\binom{m}{2}-5$ & $2$  \\ 
  3 & $0$ & $\frac{1}{3}\binom{m}{2}-5$ & $2$  \\ 
  4 & $0$ & $\frac{1}{3}\binom{m}{2}-7$ & $3$  
\end{tabular}
\end{center}

When $r=1$, by Theorem~\ref{l:23-hypergraph} a linear 3-space with $m$ vertices and $\frac{1}{3}\binom{m}{2}$ edges exists and a linear hypergraph $H$ with  the properties specified by the table can be obtained from this space by deleting an edge.
When $r \in \{2,3,4\}$ consider the appropriate linear $\{2,4\}$-hypergraph $H'$ on $m$ vertices depicted in Figure~\ref{F:UMinusExceptional}. It can be seen that the defect $D'$ of this hypergraph is an even graph with $|E(D')| \equiv 0 \mod {3}$ and each vertex having degree at least $m-7$. Thus, by Theorem~\ref{L:3HypergraphExistence}, there is a linear $3$-hypergraph $H''$ whose underlying graph is $D'$ provided that $m$ is sufficiently large.
Let $H$ be the linear hypergraph formed by taking the union of $H'$ and $H''$ and deleting $j$ edges of size 2, where $j=2$ if $r=2$ and $j=3$ if $r \in \{3,4\}$. It can be checked that $|E(H_2)|$, $|E(H_3)|$ and $|E(H_4)|$ are as specified in the table.
\end{proof}

\begin{figure}[H]
\begin{center}
\includegraphics[width=\textwidth]{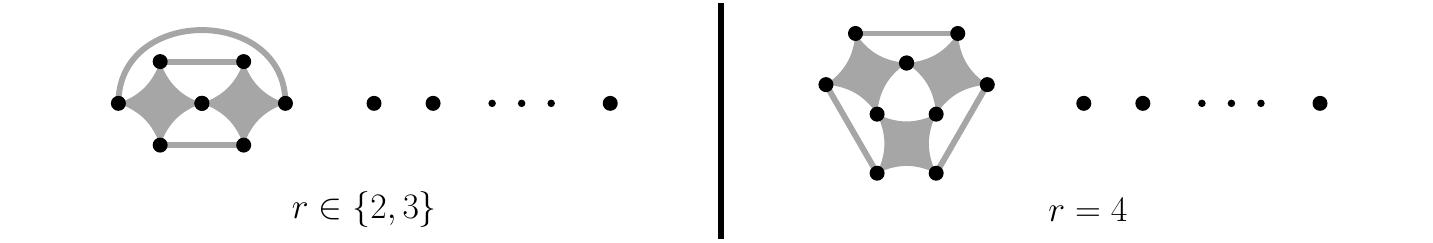}
\caption{The linear $\{2,4\}$-hypergraph $H'$ on $m$ vertices in the proof of Lemma~\ref{L:not2Mod3ExceptionalZValues}.}\label{F:UMinusExceptional}
\end{center}
\end{figure}


\subsection*{Achieving $\lfloor U^- \rfloor$ or $\lfloor U^- \rfloor-1$ when $m \equiv 2 \mod{3}$}

Recall again that $U^-=\tfrac{1}{6}m(m-1)+2n$. Finally we treat the cases of Theorem~\ref{T:belowSTS} where $m \equiv 2 \mod{3}$. The cases where $n$ is far from the triple system threshold follow from a special case of \cite[Theorem~1.3]{CheHorMam}.

\begin{lemma}[\cite{CheHorMam}]\label{T:CHMexact}
We have $Z_{2,2}(m,n) = \lfloor U^- \rfloor$ for each sufficiently large integer $m$ and for all integers $n$ such that
\[\tfrac{1}{6}\tbinom{m}{2}+\tfrac{m}{3}+40 \leq n \leq
\left\{
  \begin{array}{ll}
    \tfrac{1}{3}\tbinom{m}{2}-28 & \hbox{if $m$ is odd} \\
    \tfrac{1}{3}\tbinom{m}{2}-\frac{m}{4}-28 & \hbox{if $m$ is even.}
  \end{array}
\right.
\]
\end{lemma}


So it remains to prove Theorem~\ref{T:belowSTS} for values of $n$ close to the triple system threshold. We do this in Lemmas~\ref{L:5mod6ZValue} and \ref{L:2mod6ZValue} for the cases where $m \equiv 5 \mod{6}$ and $m \equiv 2 \mod{6}$ respectively. To assist in the proof of Lemma~\ref{L:5mod6ZValue}, we first establish the existence of a particular class of linear $\{2,4\}$-hypergraphs.

\begin{lemma}\label{L:make4Hypergraph}
Let $k \geq 5$ be an integer. For each integer $m \geq 2k+1$ there is a linear $\{2,4\}$-hypergraph on $m$ vertices with $k$ edges of size $4$ and one edge of size $2$ in which each vertex is incident to either zero or two edges.
\end{lemma}

\begin{proof}
Consider the hypergraph on vertex set $\mathbb{Z}_{2k} \cup \{\frac{3}{2}\}$ whose edges are
\[\{0,1,3,4\},\{2,3,5,6\},\{4,5,7,8\},\ldots,\{2k-6,2k-5,2k-3,2k-2\},\{2k-4,2k-3,2k-1,0\},\]
along with $\{2k-2,2k-1,1,\frac{3}{2}\}$ and $\{2,\frac{3}{2}\}$. Figure~\ref{F:make4Hypergraph} gives a depiction of this hypergraph in the case $k=5$. Since $k \geq 5$ it is routine to check that this hypergraph is linear and that each of its vertices is incident with exactly 2 edges. The required linear $\{2,4\}$-hypergraph on $m$ vertices can be formed by adding isolated vertices if required.
\end{proof}

\begin{figure}[H]
\begin{center}
\includegraphics[width=\textwidth]{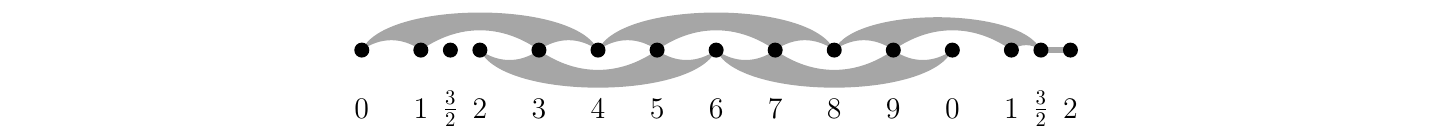}
\caption{The linear $\{2,4\}$-hypergraph on vertex set $\mathbb{Z}_{2k} \cup \{\frac{3}{2}\}$ in the proof of Lemma~\ref{L:make4Hypergraph} in the case $k=5$. Note that the vertices $0$, $1$, $\frac{3}{2}$ and $2$ are each depicted twice to allow for a clean presentation.}\label{F:make4Hypergraph}
\end{center}
\end{figure}

\begin{lemma}\label{L:5mod6ZValue}
For each sufficiently large integer $m$ such that $m\equiv 5\pmod{6}$ and for $n=\lfloor\frac{1}{3}\binom{m}{2}\rfloor-r$ where $r \in \{0,\ldots,27\}$, we have
\[Z_{2,2}(m,n) =
\left\{
  \begin{array}{ll}
    \left\lfloor U^-\right\rfloor-1 & \hbox{if $r\in \{0,1,3 \}$} \\[1.5mm]
    \left\lfloor U^- \right\rfloor & \hbox{if $r \in \{2\} \cup \{4,\ldots,27\}$.}
  \end{array}
\right.
\]
\end{lemma}
\begin{proof}
Note that $\lfloor\frac{1}{3}\binom{m}{2}\rfloor=\frac{1}{3}\binom{m}{2}-\frac{1}{3}$ since $m\equiv 5\pmod{6}$. By Lemmas~\ref{L:basicUpperBounds} and \ref{L:strongUpperBounds}, it suffices to construct a hypergraph with $m$ vertices, $n$ edges and total degree $\lfloor U^-\rfloor-1$ when $r \in \{0,1,3\}$ and $\lfloor U^-\rfloor$ otherwise. Note that
\[\left\lfloor U^- \right\rfloor=2n+\tfrac{1}{3}\tbinom{m}{2}-\tfrac{1}{3} = \tbinom{m}{2}-2r-1.\]
We will show there exists a linear $\{2,3,4,5\}$-hypergraph $H$ on $m$ vertices with $|E(H_2)|$, $|E(H_3)|$, $|E(H_4)|$ and $|E(H_5)|$ as specified in the following table. In each case simple calculations show that such an $H$ will have $n=\frac{1}{3}\binom{m}{2}-\frac{1}{3}-r$ edges and the appropriate total degree, as required to complete the proof.

\begin{center}
\begin{tabular}{c||c|c|c|c}
  \rule{4mm}{0mm}$r$\rule{4mm}{0mm} & \rule{4mm}{0mm}$|E(H_2)|$\rule{4mm}{0mm} & \rule{4mm}{0mm}$|E(H_3)|$\rule{4mm}{0mm} & \rule{4mm}{0mm}$|E(H_4)|$\rule{4mm}{0mm} & \rule{4mm}{0mm}$|E(H_5)|$\rule{4mm}{0mm} \\
  \hline
  0 & $1$ & $\frac{1}{3}\binom{m}{2}-\frac{4}{3}$ & $0$ & $0$ \\
  1 & $0$ & $\frac{1}{3}\binom{m}{2}-\frac{4}{3}$ & $0$ & $0$ \\
  2 & $0$ & $\frac{1}{3}\binom{m}{2}-\frac{10}{3}$ & $0$ & $1$ \\
  3 & $3$ & $\frac{1}{3}\binom{m}{2}-\frac{34}{3}$ & $5$ & $0$ \\
  $\geq 4$ & $1$ & $\frac{1}{3}\binom{m}{2}-\frac{7}{3}-2r$ & $r+1$ & $0$
\end{tabular}
\end{center}

\smallskip

When $r \in \{0,1\}$, a linear hypergraph $H$ with the properties specified by the table can be obtained by using Theorem~\ref{l:23-hypergraph} to obtain a linear $\{2,3\}$-space with $\frac{1}{3}\binom{m}{2}-\frac{4}{3}$ edges of size 3 and four edges of size 2, and then deleting some edges of size 2. When $r=2$, a linear hypergraph $H$ with the properties specified by the table exists by applying Theorem~\ref{3-GDD} with $g=5$, $h=1$ and $w=m-5$. When $r \geq 4$, provided that $m$ is sufficiently large, by Lemma~\ref{L:make4Hypergraph} there is a linear $\{2,4\}$-hypergraph $H'$ with $m$ vertices, $|E(H'_2)|=1$, $|E(H'_4)|=r+1$ and in which each vertex is incident with zero or two edges. It can be seen that the defect $D'$ of this hypergraph is an even graph with $|E(D')| \equiv 0 \mod {3}$ and each vertex having degree at least $m-7$. Thus, by Theorem~\ref{L:3HypergraphExistence}, there is a linear $3$-hypergraph $H''$ whose underlying graph is $D'$ provided that $m$ is sufficiently large. Let $H$ be the linear hypergraph formed by taking the union of $H'$ and $H''$. It can be checked that $|E(H_2)|$, $|E(H_3)|$, $|E(H_4)|$ and $|E(H_5)|$ are as specified in the table.
When $r=3$, note that a linear hypergraph with the properties specified by the table can be formed from the required hypergraph for $r=4$ by deleting an edge of size 3 and adding two distinct edges of size 2 that are subsets of the deleted edge.
\end{proof}

\begin{lemma}\label{L:2mod6ZValue}
For each sufficiently large integer $m$ such that $m\equiv 2\pmod{6}$ and for $n = \lfloor\frac{1}{3}\binom{m}{2}-\frac{m}{4}\rfloor-r$ where $r \in \{0,\ldots ,27\}$ we have
$Z_{2,2}(m,n) =\left\lfloor U^- \right\rfloor$.
\end{lemma}


\begin{proof}
In view of Lemma~\ref{L:basicUpperBounds}, it suffices to find a linear hypergraph with $m$ vertices, $n$ edges, and total degree $\lfloor U^- \rfloor$. We will show there exists a linear $\{2,3,4\}$-hypergraph $H$ on $m$ vertices with $|E(H_2)|$, $|E(H_3)|$ and $|E(H_4)|$ as specified in the following table. In each case simple calculations show that such an $H$ will have $n$ edges and total degree $\lfloor U^- \rfloor = \frac{1}{3}\binom{m}{2}-\frac{1}{3}+2n$, as required to complete the proof.

\begin{center}
\begin{tabular}{ll||c|c|c}
  \multicolumn{2}{c||}{case} & \rule{2mm}{0mm}$|E(H_2)|$\rule{2mm}{0mm} & \rule{2mm}{0mm}$|E(H_3)|$\rule{2mm}{0mm} & \rule{2mm}{0mm}$|E(H_4)|$\rule{2mm}{0mm} \\ 
  \hline
  $m \equiv 2 \mod{12}$ & $n=\frac{1}{3}\binom{m}{2}-\frac{m}{4}-\frac{5}{6}-r$ & $1$ & $\frac{1}{3}\binom{m}{2}-\frac{m}{2}-\frac{10}{3}-2r$ & $\frac{m}{4}+\frac{3}{2}+r$ \\ 
  $m \equiv 8 \mod{12}$ & $n=\frac{1}{3}\binom{m}{2}-\frac{m}{4}-\frac{1}{3}-r$  & $1$ & $\frac{1}{3}\binom{m}{2}-\frac{m}{2}-\frac{7}{3}-2r$ & $\frac{m}{4}+1+r$ \\ 
\end{tabular}
\end{center}

To see that such a hypergraph exists, consider the linear $\{2,4\}$-hypergraph $H'$ on $m$ vertices depicted in Figure~\ref{F:UMinusExceptional2} where $b$ is chosen to be $2r+4$ if $m \equiv 2 \mod{12}$ and chosen to be $2r+3$ if $m \equiv 8 \mod{12}$. Note that $H'$ has $3b+\frac{m-10b-2}{4}=\frac{2b+m-2}{4}$ edges of size 4 and that this agrees with the number in the fourth column of the table. Since $\binom{m}{2} \equiv 1 \mod{3}$, it can be seen that the defect $D'$ of this hypergraph is an even graph with $|E(D')| \equiv 0 \mod {3}$ and each vertex having degree at least $m-10$. Thus, by Theorem~\ref{L:3HypergraphExistence}, there is a linear $3$-hypergraph $H''$ whose underlying graph is $D'$ provided that $m$ is sufficiently large.  Let $H$ be the linear hypergraph formed by taking the union of $H'$ and $H''$. It can be checked that $|E(H_2)|$, $|E(H_3)|$ and $|E(H_4)|$ are as specified in the table.
\end{proof}

\begin{figure}[H]
\begin{center}
\includegraphics[width=\textwidth]{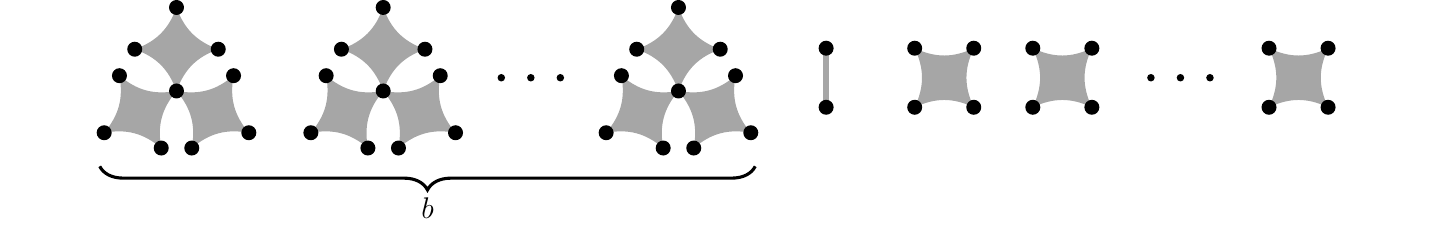}
\caption{The linear $\{2,4\}$-hypergraph $H'$ on $m$ vertices in the proof of Lemma~\ref{L:2mod6ZValue}.}\label{F:UMinusExceptional2}
\end{center}
\end{figure}

\subsection*{Proof of Theorem~\ref{T:belowSTS}}

We can now prove Theorem~\ref{T:belowSTS} by combining our previous results.

\begin{proof}[\textup{\textbf{Proof of Theorem~\ref{T:belowSTS}.}}]
It is routine to check that each pair $(m,n)$ of parameters covered by Theorem~\ref{T:belowSTS} falls into exactly one of the cases given in the first column of Table~\ref{Tab:belowSTS}. For each case, Theorem~\ref{T:belowSTS} claims that the value of $Z_{2,2}(m,n)$ is as given in the second column, and the result listed in the third column establishes that this is indeed the case.
\end{proof}

\section{Conclusion}

In this paper we concentrated on determining $Z_{2,2}(m,n)$ for values of $m$ and $n$ near the triple system threshold. The results presented here could be extended in various ways. Firstly, one could aim to prove a version of Theorem~\ref{T:belowSTS} that covered all values of $m$ rather than simply large values. Secondly, the methods employed here could also be applied to find analogous results for $Z_{2,t}(m,n)$ for $t > 3$ (see \cite{CheHorMam} for indications of how our methods can be adapted to cases where $t>2$). Finally, for any fixed value of $k \geq 4$, one could also attempt to extend our results here to consider values of $Z_{2,2}(m,n)$ for values of $m$ and $n$ near the linear $k$-space threshold of $n=\binom{n}{2}/\binom{k}{2}$. However, we believe that each of these endeavours would require substantially more intricate arguments or significantly more case analysis or both.

\subsection*{Acknowledgements}

Some of this research was undertaken while the first author visited Monash University. This visit was supported by Henan Normal University. The second author was supported by the Australian Research Council (grant DP220102212). He thanks Donald Knuth who in 2015 asked him a question about the value of $Z_{2,2}(m,\frac{1}{3}\binom{m}{2})$ for $m \equiv 0,4 \mod{6}$ that eventually led to this line of research (cf. \cite[Exercise 7.2.2.2--488]{Knu} and its solution). He also thanks Charles Colbourn for useful discussions.

\end{document}